\documentclass[reqno, dvipsnames, oneside]{amsart}
\usepackage[utf8]{inputenc}
\usepackage{amscd, amssymb, amsthm, mathrsfs, amsmath, amsrefs, amsfonts, graphicx}
\usepackage{graphicx}
\usepackage{xcolor}
\usepackage[a4paper, total={5in, 8in}]{geometry}
\usepackage{hyperref}
\usepackage{amssymb}
\usepackage{todonotes}
\usepackage{stmaryrd}

\newcommand{\Cliff}{\mathrm{Cliff}}
\newcommand{\bR}{\mathbb{R}}
\newcommand{\bN}{\mathbb{N}}
\newcommand{\Z}{\mathbb{Z}}
\newcommand{\C}{\mathbb{C}}
\newcommand{\cB}{\mathcal{B}}
\newcommand{\cC}{\mathcal{C}}
\newcommand{\cK}{\mathcal{K}}
\newcommand{\cH}{\mathcal{H}}
\newcommand{\cS}{\mathcal{S}}
\newcommand{\cU}{\mathcal{U}}

\newcommand{\hodot}{\widehat{\odot}}

\newcommand{\fS}{\mathfrak{S}}

\newcommand{\hato}{\widehat{\otimes}}

\usepackage[all]{xy}

\newtheorem{theorem}{Theorem}[section]
\newtheorem{lemma}[theorem]{Lemma}
\newtheorem{proposition}[theorem]{Proposition}

\newtheorem*{corollary*}{Corollary}

\theoremstyle{definition}

\newtheorem{definition}[theorem]{Definition}

\theoremstyle{remark}
\newtheorem{example}[theorem]{Example}
\newtheorem{remark}[theorem]{Remark}

\usepackage{xargs}

\title{On the Baum-Connes conjecture for $D_{\infty}$}

\author[1]{Eugenia Ellis}
\author[2]{Emanuel Rodr\'iguez Cirone}
\author[3]{Gisela Tartaglia}
\thanks{The first author was partially supported by ANII, CSIC and PEDECIBA.  The first and third authors were partially supported by grant PICT-2021-I-A-0 0710. The third author was partially supported by grant PIP 11220200100423CO}
\address[1]{IMERL, Fac. de Ingenier\'ia, Universidad de la República, Montevideo, Uruguay}
\address[2]{Área de Matemática - CBC - UBA, Buenos Aires, Argentina}
\address[3]{CMaLP - CONICET, FCE-UNLP, La Plata, Argentina}
\email[1]{eellis@fing.edu.uy}
\email[2]{erodriguezcirone@cbc.uba.ar}
\email[3]{gtartaglia@mate.unlp.edu.ar}
\date{}

\begin{document}
\begin{abstract}
         We make an exposition of the proof of the Baum-Connes conjecture for the infinite dihedral group following the ideas of Higson and Kasparov.
\end{abstract}

\maketitle

\section{Introduction}

The $K$-theory of $C^*$-algebras is a growing area of research with applications to index theory, noncommutative geometry and classification of $C^*$-algebras.
A useful tool for computing the $K$-theory of group $C^*$-algebras is the Baum-Connes conjecture, introduced in \cite{bch} using Kasparov's $KK$-theory. For a countable discrete group $G$ and a $G$-$C^*$-algebra $A$, the conjecture states that certain assembly map
\[\mu_r:K^{top}(G,A)\to K(C_r^*(G,A))\]
is an isomorphism, where the left hand side is defined using $KK$-groups.
A different assembly map was formulated by Davis and L\"uck \cite{dl} replacing the left hand side with a homotopy theoretic construction. Both assembly maps were shown to be equivalent in \cite{kranz}. The conjecture was proved for a large class of groups though it is known not to hold for general $G$ and $A$; see \cite{hls}. For example, the conjecture was proved for a-T-menable groups \cite{HK}, hyperbolic groups \cite{laf} and one-relator groups \cite{bba}. For a comprehensive and up-to-date exposition on the Baum-Connes conjecture, we refer the reader to the survey \cite{GAJV}.

One of the main contributions regarding the Baum-Connes conjecture is the work of Higson and Kasparov, who proved in \cite{HK} that the conjecture holds for a-T-menable groups (i.e. groups that act affine, properly and isometrically on a Hilbert space). In their work, Higson and Kasparov used asymptotic morphisms and described the left hand side of the assembly map in terms of $G$-equivariant $E$-theory groups \cite{GHT}. The paper \cite{HK} is technically involved and its main ideas are carefully explained in \cite{libro-higson-connes}, where the conjecture for finite groups and for $\Z^n$ is discussed before approaching the general case. In this work we go one step further in the list of examples and expose the proof of the conjecture for the infinite dihedral group $D_\infty$. Following the arguments given in \cite{libro-higson-connes}, we use the fact that $D_\infty$ acts on a finite-dimensional Euclidean space to do some calculations explicitly and to avoid part of the technical machinery of \cite{HK}.

\subsection*{\it Acknowledgements} The authors wish to thank the organizers of the workshop {\it Matemáticas en el Cono Sur 2}, where this project was initiated, and Natalia Pacheco-Tallaj for having participated in the project in its early stages. The last two authors also thank Eugenia Ellis for her hospitality and support during their visits to the IMERL--UdelaR in Montevideo.

\section{Preliminaries}\label{sec:preliminaries}

In this section we gather from \cite{libro-higson-connes} relevant definitions and results needed to state and prove the Baum-Connes conjecture with coefficients for the group $D_\infty$.

\subsection{Graded \texorpdfstring{$G$-$C^*$-algebras}{G-C*-algebras}}
Let $A$ be a $C^*$-algebra. A \emph{grading} on $A$ is given by a $*$-homomorphism $\alpha:A\to A$ such that $\alpha^2=1$. Equivalently, a grading is given by two $*$-subspaces $A_0$ and $A_1$ satisfying $A=A_0\oplus A_1$, and $A_iA_j\subseteq A_{i+j}$ (mod 2). Elements of $A_0$ (those $a\in A$ such that $\alpha(a)=a$) are called even graded, and elements of $A_1$ (those $a\in A$ such that $\alpha(a)=-a$) are called odd graded. Throughout this article, all $C^*$-algebras will be graded.

\begin{example}
    Let $\cS$ denote the $C^*$-algebra $C_0(\bR)$ of continuous, complex valued functions on $\bR$ that vanish at infinity, with grading operator given by $f(x)\mapsto f(-x)$. This grading induces the decomposition $\cS = \left\{ \text{even functions} \right\} \oplus \left\{ \text{odd functions} \right\}$.
\end{example}

\begin{example}
    A graded Hilbert space is a Hilbert space $\cH$
equipped with an orthogonal decomposition $\cH=\cH_0\oplus\cH_1$. This grading induces a grading on the $C^*$-algebra $\cB(\cH)$ of bounded operators on $\cH$ as follows. First note that every $T\in \cB(\cH)$ can be identified with a 2 by 2 matrix, then declare the diagonal matrices to be even, and the off-diagonal ones to be odd.
\end{example}

Throughout this paper, $G$ is a countable discrete group. A \emph{graded $G$-$C^*$-algebra} is a $C^*$-algebra $A$ equipped with an action of $G$ by grading-preserving $*$-automorphisms.

\subsection{Maximal tensor product}
Let $A$ and $B$ be graded $C^*$-algebras and let $A\hodot B$ be the algebraic tensor product of the underlying vector spaces. Endow $A\hodot B$ with the multiplication, involution and grading given by the following formulas on elementary tensors of homogeneous elements:
\begin{align*}
    (a_1\hodot b_1)(a_2\hodot b_2)&=(-1)^{\partial b_1\partial a_2}a_1a_2\hodot b_1b_2 \\
    (a\hodot b)^*&=(-1)^{\partial a\partial b}a^*\hodot b^* \\
    \partial (a\hodot b)&=\partial a+\partial b\quad\text{(mod $2$)}
\end{align*}
Here, $\partial a=0$ for $a\in A_0$ and $\partial a=1$ for $a\in A_1$. The \emph{maximal graded tensor product} $A\hato B$ is the completion of $A\hodot B$ with respect to the maximal norm; see for example \cite{libro-higson-connes}*{Def. 1.9}. It has the following universal property: if $C$ is a graded $C^*$-algebra and if $f:A\to C$ and $g:B\to C$ are graded $*$-homomorphisms whose images graded-commute, then there exists a unique graded $*$-homomorphism $A\hato B\to C$ that maps $a\hato b$ to $f(a)g(b)$. Moreover, if $f:A\to C$ and $g:B\to D$ are graded $*$-homomorphisms, then there exists a unique graded $*$-homomorphism $A\hato B\to C\hato D$ that maps $a\hato b$ to $f(a)\hato g(b)$.

\subsection{Crossed products}
Let $A$ be a graded $G$-$C^*$-algebra and let $C_c(G,A)$ be the linear space of finitely supported, $A$-valued functions on $G$. Then $C_c(G,A)$ is a graded involutive algebra with convolution multiplication. The involution is defined by $f^*(g)=g\cdot (f(g^{-1})^*)$ for $f\in C_c(G,A)$ and $g\in G$ and the grading automorphism acts pointwise. The \emph{full crossed product graded $C^*$-algebra} $C^*(G,A)$ is the completion of $C_c(G,A)$ in the smallest $C^*$-norm that makes all the covariant representations continuous; see \cite{libro-higson-connes}*{Def. 2.19}. The \emph{reduced crossed product graded $C^*$-algebra} $C_r^*(G,A)$ is the image of $C^*(G,A)$ in the regular representation on $\ell^2(G,A)$; see \cite{libro-higson-connes}*{Def. 2.21}. It is well known that $C^*(G,A)=C^*_r(G,A)$ if $G$ is amenable; see for example \cite{dana}*{Sec. 7.2}.

\subsection{Asymptotic morphisms} Let $A$ and $B$ be graded $C^*$ -algebras. An \emph{asymptotic morphism} $\varphi:A\dashrightarrow B$ is a family of functions $\{\varphi_t:A\to B\}_{t\geq 1}$ such that the function $[1,\infty)\to B$, $t\mapsto \varphi_t(a)$, is continuous and bounded for every $a\in A$, and such that the following asymptotic conditions are satisfied:
\begin{align*}
            \varphi_t(a_1a_2)-\varphi_t(a_1)\varphi_t(a_2) &\xrightarrow[t\to\infty]{} 0\cr
\varphi_t(a_1+a_2)- \varphi_t(a_1)-\varphi_t(a_2) &\xrightarrow[t\to\infty]{} 0\cr
\varphi_t(\lambda a_1)-\lambda\varphi_t(a_1)&\xrightarrow[t\to\infty]{} 0\cr
\varphi_t(a_1^*)- \varphi_t(a_1)^*&\xrightarrow[t\to\infty]{} 0\cr
\alpha_B(\varphi_t(a_1))-\varphi_t(\alpha_A(a_1))&\xrightarrow[t\to\infty]{} 0
\end{align*}
Here, $a_1,a_2\in A$, $\lambda\in \C$, and $\alpha_A, \alpha_B$ are the grading morphisms of $A$ and $B$ respectively.
Note that any graded $*$-homomorphism $\phi:A \to B$ determines a constant asymptotic morphism $\phi_t=\phi$ for all $t \geq 1$.

If $A$ and $B$ are graded $G$-$C^*$-algebras, an \emph{equivariant asymptotic morphism} $\varphi:A\dashrightarrow B$ is an asymptotic morphism $\varphi$ such that $\varphi_t(g\cdot a)-g\cdot \varphi_t(a)\to 0$ as $t\mapsto \infty$, for all $a\in A$ and all $g\in G$.

Two (equivariant) asymptotic morphisms $\varphi^0,\varphi^1: A\dashrightarrow B$ are called:
\begin{itemize}
    \item \emph{asymptotically equivalent} if
         $\lim_{t\to\infty}\lVert \varphi^0_t(a)-\varphi^1_t(a)\rVert_B= 0$
        for every $a\in A$;
    \item \emph{homotopy equivalent} if there exists an asymptotic morphism $\varphi:A\dashrightarrow C([0,1],B)$ such that $\varphi(a)(0)=\varphi^0(a)$ and $\varphi(a)(1)=\varphi^1(a)$ for every $a\in A$.
\end{itemize}
We write $\llbracket A, B\rrbracket$ for the set of homotopy classes of asymptotic morphisms from $A$ to $B$. If $A$ and $B$ are graded $G$-$C^*$-algebras, we write $\llbracket A,B\rrbracket^G$ for the set of homotopy classes of equivariant asymptotic morphisms from $A$ to $B$.

\subsection{\texorpdfstring{E-Theory}{$E$-Theory} groups}

Let $A$ and $B$ be separable graded $C^*$-algebras. Put
\[E(A,B):=\llbracket\cS\hato A\hato \cK(\cH), B\hato\cK(\cH)\rrbracket\]
where $\cK(\cH)\subset \cB(\cH)$ is the subalgebra of compact operators.
These sets $E(A,B)$ equipped with the sum induced by the direct sum of asymptotic morphisms are indeed abelian groups, see \cite{libro-higson-connes}*{Lemma 2.1}. They depend contravariantly on $A$ and covariantly on $B$ with respect to graded $*$-homomorphisms. There exists a bilinear composition law $E(A,B)\otimes E(B,C)\to E(A,C)$ that makes the groups $E(A,B)$ into the hom-sets of an additive category whose objects are separable $C^*$-algebras. Moreover, we can recover $K$-theory groups from $E$-theory since we have $E(\C,A)\cong K(A)$ for any separable $C^*$-algebra $A$, where $K$ stands for the $K$-theory of graded $C^*$-algebras.

To define equivariant $E$-theory groups, let $\cH_G$ be the infinite Hilbert space direct sum:
\[\cH_G=\bigoplus_{n=0}^{\infty}\ell^2(G)\]
This Hilbert space is equipped with the regular representation of $G$ on each summand and graded so that the even numbered summands are even and the odd numbered summands are odd. For graded separable $G$-$C^*$-algebras $A$ and $B$ put:
\[E_G(A,B):=\llbracket \cS\hato A\hato \cK(\cH_G),B\hato \cK(\cH_G) \rrbracket^G\]
These sets $E_G(A,B)$ are the hom-sets of an additive category whose objects are the graded separable $G$-$C^*$-algebras.
There is a descent functor from the $G$-equivariant $E$-theory category to the $E$-theory category that sends a $G$-$C^*$-algebra $A$ to the maximal crossed product $C^*(G,A)$; see \cite{libro-higson-connes}*{Theorem 2.13}.

The following definition will be useful later on.

\begin{definition}
    Let $\cU$ be a separable, $\Z_2$-graded Hilbert space equipped with a family of unitary $G$-actions parametrized by $t\in [1,\infty)$. This family induces a family of actions on $\cB(\cU)$ by conjugation:
    \[(g\cdot_tT)(u)= g\cdot_t(T(g^{-1}\cdot_t u)).\]
    We call the family a \emph{continuous family of $G$-actions} if for every $g\in G$ and every $T\in \cK(\cU)$, the map $t\mapsto g\cdot_t T$ is norm continuous in $t$. Suppose now that $A$ and $B$ are $G$-$C^*$-algebras, and $\phi:\cS\hato A \dashrightarrow B\hato\cK(\cU)$ is an asymptotic morphism. We say that $\phi$ is \emph{equivariant with respect to the given family of $G$-actions} if
    \[ \lim_{t\to\infty} \lVert \phi_t(g\cdot x)- g\cdot_t (\phi_t(x))\rVert = 0.\]
\end{definition}

\begin{remark}
    By \cite{libro-higson-connes}*{Remark 2.6}, if $\phi:\cS\hato A \dashrightarrow B\hato\cK(\cU)$ is equivariant with respect to a continuous family of $G$-actions, then $\phi$ determines a class in $E_G(A,B)$.
\end{remark}

\subsection{The Baum-Connes assembly map}Let $D$ be a separable $G$-$C^*$-algebra. The \emph{topological $K$-theory} of $G$ with coefficients in $D$ is defined by
$$ K^{top}(G,D)=\displaystyle\lim_{\rightarrow} E_G(C_0(X),D)$$
where the limit is taken over the collection of $G$-invariant and $G$-compact subspaces $X$ of the universal proper $G$-space $\mathcal{E}G$ (\cite{libro-higson-connes}*{Section 2.12}). 

The (full) \emph{Baum-Connes assembly map} with coefficients in $D$ is the map

\begin{equation}\label{assembly}
\mu: K^{top}(G,D)\to K(C^*(G,D))
\end{equation}
which is obtained as a limit of compositions
$$E_G(C_0(X),D)\xrightarrow{desc} E(C^*(G,C_0(X)),C^*(G,D))\xrightarrow{[p]}E(\C,C^*(G,D))$$
 of the descent homomorphism and the homomorphism induced by the class of a projection associated to a cutoff function (for details see \cite{libro-higson-connes}*{Section 2.14}).
Composing the assembly $\mu$ with the map from $K(C^*(G,D))$ to $K(C^*_r(G,D))$ induced by the surjective homomorphism $C^*(G,D)\rightarrow C^*_r(G,D)$, one obtains the (reduced) \emph{Baum-Connes assembly map} with coefficients in D:
\begin{equation*}
\mu_r: K^{top}(G,D)\to K(C^*_r(G,D)).
\end{equation*}

The \emph{Baum-Connes conjecture} (with coefficients) asserts that the assembly map $\mu_r$ is an isomorphism for every separable $G$-$C^*$-algebra $D$.

Note that for finite $G$ the conjecture is true, as it is equivalent to a well-known result of Green and Julg (see \cite{Green} and \cite{Julg}). An important tool for the study of the conjecture is the notion of proper algebra:

A $G$-$C^*$-algebra $B$ is called \emph{proper} if there exists a locally compact proper $G$-space $Z$, and an equivariant $*$-homomorphism $\phi$ from $C_0(Z)$ into the grading-degree zero part of the center of the multiplier algebra of $B$, such that $\phi(C_0(Z))\cdot B$ is norm-dense in $B$.

For a proper $G$-$C^*$-algebra $B$ the full and reduced crossed product $C^*(G,B)$ and $C^*_r(G,B)$ coincide, and the Baum-Connes conjecture is true (\cite{GHT}*{Thm.13.1}).

For general coefficient algebras the following theorem provides a strategy for studying the conjecture:

\begin{theorem}\cite{libro-higson-connes}*{Thm. 2.20}
    Let $G$ be a countable discrete group. Suppose there exists a proper $G$-$C^*$-algebra $B$ and elements $\beta \in E_G(\C,B)$ and $\alpha \in E_G(B,\C)$ such that 
    $$\alpha\circ\beta\in E(\C,\C).$$
    Then the Baum-Connes assembly map $\mu: K^{top}(G,D)\to K(C^*(G,D))$ is an isomorphism for every separable $G$-$C^*$-algebra $D$.
\end{theorem}

The proof of this result is based on the commutativity of the diagram 
\[\xymatrix{K^{top}(G,\C\hato D)\ar[r]^{\mu}\ar[d]_{\beta_*}& K(C^*(G,\C\hato D))\ar[d]^{\beta_*}\\K^{top}(G,B\hato D)\ar[r]^{\mu}_{\cong}\ar[d]_{\alpha_*}& K(C^*(G,B\hato D))\ar[d]^{\alpha_*}\\K^{top}(G,\C\hato D)\ar[r]^{\mu}& K(C^*(G,\C\hato D))\\}\]
and the fact that for a proper $G$-$C^*$-algebra $B$, $B\hato D$ is again proper for every separable $G$-$C^*$-algebra $D$.


\section{The Baum-Connes conjecture for \texorpdfstring{$D_\infty$}{the infinite dihedral group}}

In this section we revisit the proof the Baum-Connes conjecture given in  \cite{HK} for the group $D_\infty
=\langle \rho, \sigma \mid \sigma^2=1,\  \sigma\rho\sigma=\rho^{-1}\rangle$. Note that $D_{\infty}$ is an amenable group, since it is elementary amenable, and this implies $C^*(D_{\infty},A)=C^*_r(D_{\infty},A)$ for every $D_{\infty}$-$C^*$-algebra $A$. We start by constructing a proper $D_\infty$-$C^*$-algebra $\cC(\bR)$.

The group $D_\infty$ acts on $\bR$ on the left by $\sigma\cdot x= -x$ and $\rho\cdot x=x-1$. This action is affine and metrically proper. The latter means that for every $R>0$ and for every $x\in \bR$, there are only finitely many $g\in D_\infty$ with $|x-g\cdot x|\leq R$ (note that $|x-\rho^l\sigma\cdot x|\geq ||l|-|2x||$, $\forall l\in\Z$).

Let $\operatorname{Cliff}(\bR)$ be the complexified Clifford algebra of $\bR$. It can be identified with the unital $\Z_2$-graded $C^*$-algebra $\C\oplus\C$. The homogeneous elements of grading-degree one are those of the form $(z,-z)$, and the grading-degree zero ones are of the form $(w,w)$. 

The morphism $\pi:D_\infty\to \{\pm 1\}$ determined by $\pi(\rho)=1$ and $\pi(\sigma)=-1$ is an orthogonal representation of $D_\infty$ that induces an action of $D_\infty$ on $\Cliff(\bR)$. If we write $\Cliff(\bR)=\C\oplus\C$ we have $\pi(\rho)(z,w)=(z,w)$ and $\pi(\sigma)(z,w)=(w,z)$.
We will write $\cC(\bR)$ for the $C^*$-algebra $C_0(\bR,\operatorname{Cliff}(\bR))=C_0(\bR)\oplus C_0(\bR)$, which is a $D_\infty$-algebra with the action given by:
$$(g\cdot h)(x)=\pi(g)( h(g^{-1}\cdot x)), \hspace{0.2cm}\forall g \in D_\infty, h\in \cC(\bR),x\in \bR.$$

\begin{lemma}
    $\cC(\bR)$ is a proper $D_\infty$-$C^*$-algebra.
\end{lemma}

\begin{proof}
    We will consider the locally compact space $\bR$. The action of $D_\infty$ on $\bR$ is proper iff for every compact subset $K$ of $\bR$, the set $\left\{ g\in D_\infty \mid g\cdot K \cap K\neq\emptyset\right\}$ is finite. Given a compact subset $K\subseteq \bR$, there exists $N\in\bN$ such that $K\subseteq [-N,N]$, and for every $l\geq 2N$, and every $x\in K$ we have $|\rho^l\cdot x|> N$.   The action of $D_\infty$ on $C_0(\bR)$ is given by:
    $$(g\cdot f)(x)=f(g^{-1}\cdot x).$$ 
    The multiplier algebra of $\cC(\bR)$ is the commutative algebra $C_b(\bR)\oplus C_b(\bR)$.
    Let $\phi:C_0(\bR)\to C_b(\bR)\oplus C_b(\bR)$ be the $*$-homomorphism given by $$\phi(f)(x)=(f(x), f(x)).$$
Note that $\phi$ is $D_\infty$-equivariant:
    \begin{align*}
     \phi(g\cdot f)(x)=&((g\cdot f)(x), (g\cdot f)(x)) \cr
     =&(f(g^{-1}\cdot x),  f(g^{-1}\cdot x))\cr
     =& \pi(g)(f(g^{-1}\cdot x),  f(g^{-1}\cdot x))\cr
     =&(g\cdot \phi(f))(x).
    \end{align*}
To see that $\phi(C_0(\bR))\cdot \cC(\bR)$ is norm-dense in $\cC(\bR)$, use the fact that $C_0(\bR)$ has an approximate unit.
\end{proof}

As explained in the previous section, to show that the assembly map \eqref{assembly} is an isomorphism it is enough to construct elements $\beta \in E_{D_{\infty}}(\C,\cC(\bR))$ and $\alpha \in E_{D_{\infty}}(\cC(\bR),\C)$ satisfying $\alpha\circ\beta=1 \in E_{D_{\infty}}(\C,\C)$.

\subsection{The element \texorpdfstring{$\beta$}{beta}}
Let $C:\bR\to\Cliff(\bR)=\C\oplus\C$ be the inclusion given by $C(x)=(x,-x)$. Since $C(x)$ is a self-adjoint element of $\Cliff(\bR)$ it makes sense to consider its continuous functional calculus. Recall the definition of the $C^*$-algebra $\cS$ from Section \ref{sec:preliminaries} and endow it with the trivial action of $D_\infty$.
For $t\geq 1$, let $\beta_t:\cS\to\cC(\bR)$ be the $*$-homomorphism given by
\[\beta_t(f)(x)=f(t^{-1}C(x))=(f(t^{-1}x), f(-t^{-1}x)).\]

\begin{lemma}\label{lem:betaequiv}
    The asymptotic morphism $\beta=\{\beta_t\}_{t\geq 1}$ is asymptotically $D_\infty$-equivariant.
\end{lemma}
\begin{proof}
    Since the action of $D_\infty$ on $\cS$ is trivial, we have to show that
  \begin{equation}\label{eq:betat}\lim_{t\to\infty}\lVert\beta_t(f)-g\cdot\beta_t(f)\rVert_{\cC(\bR)}=0\end{equation}
    for every $f\in \cS$ and $g\in D_\infty$. Let us begin with $g=\rho$. For $f\in\cS$ we have:
\begin{align*}
\lVert\beta_t(f)- \rho\cdot\beta_t(f)\rVert_{\cC(\bR)} &= \sup_{x\in\bR} \lVert\beta_t(f)(x)- \pi(\rho)(\beta_t(f)(\rho^{-1}\cdot x))\rVert_{\Cliff(\bR)}\cr
&= \sup_{x\in\bR} \lVert F(t^{-1}x)-F(t^{-1}(x+1)) \rVert_{\Cliff(\bR)}
\end{align*}
where $F\in\cC(\bR)$ is given by $F(x)=(f(x),f(-x))$.
Let $\epsilon > 0$. Since $F$ is uniformly continuous, then there exists $\delta > 0$ such that
$$ |t^{-1}|=|t^{-1}x-(t^{-1}(x+1))| < \delta \Rightarrow \lVert F(t^{-1}x)-F(t^{-1}(x+1)) \rVert < \epsilon$$
for every $x\in\bR$. This implies that
$$\lim_{t\to \infty} \lVert \beta_t(f)- \rho\cdot\beta_t(f)\rVert_{\cC(\bR)}=0. $$

For the case $g=\sigma$ we have:
\begin{align*}
\beta_t(f)(x)- (\sigma\cdot\beta_t(f))(x) &= \beta_t(f)(x)- \pi(\sigma)(\beta_t(f)(\sigma^{-1}\cdot x))\cr
&= (f(t^{-1}x),f(-t^{-1}x) )- \pi(\sigma)(f(-t^{-1}x),f(t^{-1}x) )\cr
&= 0
\end{align*}
This shows that $\lVert\beta_t(f)-\sigma\cdot\beta_t(f)\rVert_{\cC(\bR)}=0$.
Since $D_\infty$ is generated by $\rho$ and $\sigma$, it follows that \eqref{eq:betat} holds for all $g\in D_\infty$.
\end{proof}

\subsection{The element \texorpdfstring{$\alpha$}{alpha}}

Let $\cH$ be the graded Hilbert space $L^2(\bR, \operatorname{Cliff}(\bR))= L^2(\bR)\oplus L^2(\bR)$. The group $D_\infty$ acts on $\cH$ by
$$ (g\cdot F)(x)=\pi(g)(F(g^{-1}\cdot x))$$
for $g\in D_\infty$, $F\in\cH$ and $x\in\bR$.

Letting $\cB(\cH)$ be the $C^*$-algebra of bounded operators on $\cH$, we have an induced $D_\infty$-action on $\cB(\cH)$ given by
$$ (g\cdot T)(F)= g\cdot(T(g^{-1}\cdot F)),$$
for $g\in D_\infty$, $T\in\cB(\cH)$ and $F\in\cH$.

Let $\fS(\bR)\subseteq \cH$ be the space of Schwartz functions. The \emph{Dirac operator} $D$ is an unbounded operator on $\cH=L^2(\bR)\oplus L^2(\bR)$ with domain $\fS(\bR)$ defined by
\begin{align*}
    D(F_1,F_2)= &\left(\frac{dF_2}{dx},-\frac{dF_1}{dx}\right) 
\end{align*}
for all $F=(F_1,F_2) \in  \fS(\bR)$.
By \cite{libro-higson-connes}*{Lemma 1.8}, $D$ is essentially self-adjoint and we can apply functional calculus: $\forall f\in \cS$, $f(D) \in 
 \cB(\cH)$. Moreover, for $F\in \cC(\bR)$ we have $f(D)M_F \in \cK(\cH)$, where $M_F\in\cB(\cH)$ is the multiplication operator and $\cK(\cH)\subset\cB(\cH)$ is the subalgebra of compact operators.

For $F\in \cC(\bR)$ and $t\in [1,\infty)$, let $F_t \in \cC(\bR)$ be the function $F_t(x)=F(t^{-1}x)$.
By \cite{libro-higson-connes}*{Proposition 1.5}, there exists, up to equivalence, a unique asymptotic morphism $\alpha: \cS\hato \cC(\bR)\dashrightarrow \cK(\cH)$ defined as follows on the elementary tensors:
$$ \alpha_t(f\hato F)= f(t^{-1}D)M_{F_t}.$$


Let us define a continuous family of $D_\infty$-actions on $\bR$ as follows: for every $s\geq 0$ put $\rho\cdot_s x=x-s$ and $\sigma\cdot_s x =-x$.
The induced action on $\cH$ is
 $$(g\cdot_s F)(x)=\pi(g)( F(g^{-1}\cdot_s x)).$$
for $F\in\cH$ and $g\in D_\infty$. This defines a continuous family of $D_\infty$-actions on $\cK(\cH)$. 

 \begin{proposition}\label{propo:alphaequiv} The asymptotic morphism $\alpha$ is equivariant with respect to the family of actions defined above i.e. it verifies 
\[\lim_{t\to \infty}\lVert\alpha_t(f\hato g\cdot F)- g\cdot_t(\alpha_t(f\hato F)) \rVert=0\]
for every $g\in D_\infty$, $f\in\cS$ and $F\in\cC(\bR)$.
 \end{proposition}

 \begin{proof}
 We will actually show that $\lVert\alpha_t(f\hato g\cdot F)- g\cdot_t(\alpha_t(f\hato F)) \rVert=0$ for all $g$, $f$ and $F$.
     Since $D_\infty$ is generated by $\sigma$ and $\rho$, it suffices to prove that this holds for these two elements.
     
     Let us begin with the case $g=\rho$. Fix $t\in [1,\infty)$, $f\in \cS$ and $F \in \cC(\bR)$. We will show that $\alpha_t(f\hato (\rho\cdot F))=\rho\cdot_t(\alpha_t(f\hato F))$.
     On one hand we have:
     \[\alpha_t(f\hato (\rho\cdot F))= f(t^{-1}D)M_{(\rho\cdot F)_t}\]
    On the other hand, since $D$ is translation invariant, we have:
    \begin{align*}\rho\cdot_t(\alpha_t(f\hato F))&= \rho\cdot_t(f(t^{-1}D)M_{F_t})\\
    &= (\rho\cdot_tf(t^{-1}D))(\rho\cdot_tM_{F_t})\\
    &= f(t^{-1}D) (\rho\cdot_t M_{F_t})\end{align*}
    We claim that $M_{(\rho\cdot F)_t}=\rho\cdot_tM_{F_t}$. Indeed, for $G\in \cH$ and $x\in\bR$, we have:
    \[M_{(\rho\cdot F)_t}(G)(x)= (\rho\cdot F)_t(x)G(x)=(\rho\cdot F)(t^{-1}x)G(x)= F(t^{-1}x+1)G(x)\]
    \begin{align*}
    (\rho\cdot_tM_{F_t})(G)(x)&= (\rho\cdot_t(M_{F_t}(\rho^{-1}\cdot_t G)))(x)\\
    &=(M_{F_t}(\rho^{-1}\cdot_t G))(x+t)\\
    &=F_t(x+t) (\rho^{-1}\cdot_t G)(x+t)\\
    &= F(t^{-1}x+1)G(x)
    \end{align*}
    This proves our claim and finishes the case $g=\rho$.

    Let us now consider the case $g=\sigma$. On one hand we have:
\[ \alpha_t(f\hato \sigma\cdot F)= f(t^{-1}D)M_{(\sigma\cdot F)_t}\]
On the other, we have:
    \begin{align*}
  \sigma\cdot_t(\alpha_t(f\hato F))&= \sigma\cdot_t (f(t^{-1}D)M_{F_t})\cr
  &= (\sigma\cdot_tf(t^{-1}D))(\sigma\cdot_t M_{F_t})
    \end{align*}
   We claim that $M_{(\sigma\cdot F)_t}=\sigma\cdot_tM_{F_t}$. Indeed, for $G\in \cH$ and $x\in\bR$, we have:

   \begin{align*}
 (\sigma\cdot_tM_{F_t})(G)(x)&= \sigma\cdot_t (M_{F_t}(\sigma^{-1}\cdot_t G))(x) \cr&= \pi(\sigma)\left( M_{F_t}(\sigma^{-1}\cdot_t G)(\sigma^{-1}\cdot_t x)\right)\cr
 &= \pi(\sigma)\left(F_t(-x) (\sigma^{-1}\cdot_t G)(-x)\right)\cr
 &= \pi(\sigma)\left(F(-t^{-1}x)\pi(\sigma^{-1})\left( G(x)\right) \right)\cr
 &=[\pi(\sigma)\left(F(-t^{-1}x)\right)]G(x)
 \end{align*}

 \begin{align*}
 M_{(\sigma\cdot F)_t}(G)(x)&= (\sigma\cdot F)_t(x)G(x)\cr
 &= (\sigma\cdot F)(t^{-1}x)G(x)\cr
 &= [\pi(\sigma)\left( F(-t^{-1}x)\right)]G(x)
 \end{align*}
 Let us now show that $\sigma\cdot_tf(t^{-1}D)=f(t^{-1}D)$. Since $\cS$ is generated by $u=e^{-x^2}$ and $v=xe^{-x^2}$, it suffices to consider the cases $f=u$ and $f=v$.
In the case $f=u$, $u(t^{-1}D)$ is convolution by $w=e^{-\frac{1}{4}t^{-2}\lVert x\rVert^2}$. Unravelling the definitions and using that $w$ is an even function, one shows that
$\sigma\cdot_tu(t^{-1}D)=u(t^{-1}D)$. In the case $f=v$, we have:
\begin{align*}\sigma\cdot_t v(t^{-1}D)&=\sigma\cdot_t [t^{-1}Du(t^{-1}D)] \\
& =(\sigma\cdot_t t^{-1}D)(\sigma\cdot_t u(t^{-1}D)) \\
& =(\sigma\cdot_t t^{-1}D) u(t^{-1}D)
\end{align*}
It is easily verified from the definitions that $\sigma\cdot_t t^{-1}D=t^{-1}D$.
\end{proof}

\begin{proposition}
    Let $\alpha:\cS\hato \cC(\bR)\dashrightarrow \cK(\cH)$ and $\beta:\cS\dashrightarrow\cC(\bR)$ be the asymptotic morphisms defined above. Then we have $\alpha\circ\beta=1\in E_{D_\infty}(\C, \C)$.
\end{proposition}
\begin{proof}
    Let $s\in[0,1]$ and let $\cC_s(\bR)$ be the $C^*$-algebra $\cC(\bR)$ endowed with the $D_\infty$-action $\cdot_s$. Consider the $C^*$-algebra $\cC_{[0,1]}(\bR):=C([0,1],\cC(\bR))$ with the $D_\infty$-action
    \[(g\cdot h)(s):=g\cdot_s h(s)\]
    for $g\in D_\infty$, $h\in\cC_{[0,1]}(\bR)$ and $s\in[0,1]$. Define $\cK_s(\cH)$ and $\cK_{[0,1]}(\cH)$ in a similar fashion --- using the scaled action $\cdot_s$. 

    With this notation, we have to prove that the composition
    $$\C\xrightarrow{\beta} \cC_1(\bR) \xrightarrow{\alpha} \C$$
    is the identity in $E_{D_{\infty}}(\C,\C).$

Upon tensoring with $C[0,1]$, the asymptotic morphism $\alpha:\cS\hato \cC(\bR)\dashrightarrow \cK(\cH)$ induces an asymptotic morphism $$ \overline{\alpha}: \cS\hato \cC_{[0,1]}(\bR)\dashrightarrow \cK_{[0,1]}(\cH)$$ given by $\overline{\alpha}_t(f\otimes h)(s)=\alpha_t(f\otimes h(s))$.
 With an argument similar to the one used in Proposition \ref{propo:alphaequiv} it can be shown that $\overline{\alpha}$ determines a class in $E_{D_\infty}(\cC_{[0,1]}(\bR), C[0,1])$.

Similarly, $\beta:\cS\dashrightarrow \cC(\bR)$ induces an asymptotic morphism
\[\overline{\beta}:\cS\dashrightarrow \cC_{[0,1]}(\bR)\]
upon tensoring with $C[0,1]$ and then composing with the inclusion $\cS\subseteq \cS[0,1]$ as constant functions. The same arguments used in Lemma \ref{lem:betaequiv} show that $\overline{\beta}$ is asymptotically equivariant. Consider the following commutative diagram of equivariant $E$-theory morphisms, where $\varepsilon_s$ denotes the morphism induced by the evaluation at $s$:
\[\xymatrix{\C \ar@{=}[d]\ar[r]^-{\overline{\beta}} & \cC_{[0,1]}(\bR) \ar[d]^-{\varepsilon_s}\ar[r]^-{\overline{\alpha}} & C[0,1]\ar[d]^-{\varepsilon_s} \\
\C\ar[r]^-{\beta} & \cC_s(\bR)\ar[r]^-{\alpha} & \C}\]
Note that the asymptotic morphisms $\alpha$ and $\beta$ still define classes in equivariant $E$-theory when we replace $\cC(\bR)$ with $\cC_s(\bR)$. Since $\varepsilon_s$ is an equivariant homotopy equivalence, it induces an isomorphism on equivariant $E$-theory. Moreover, the $E$-theory class of $\varepsilon_s$ does not depend on $s$ since any $\varepsilon_s$ is a right inverse to the inclusion $\C\hookrightarrow C[0,1]$ as constant functions. It follows that the $E$-theory class of the composite
    $$\C\xrightarrow{\beta} \cC_s(\bR) \xrightarrow{\alpha} \C$$
does not depend on $s$. We will show that this is the identity for $s=0$. Note that in this case, each $\beta_t$ is an equivariant $*$-homomorphism. It follows that the equivariant asymptotic morphism $\beta$ is equivariantly homotopy equivalent to the equivariant $*$-homomorphism $\beta_1$. Using this fact, it is enough to compute the following composition:
$$\C\xrightarrow{\beta_1} \cC_0(\bR) \xrightarrow{\alpha} \C $$
By \cite{libro-higson-connes}*{Theorem 1.17} this composite is asymptotically equivalent to the asymptotic morphism $\gamma:\cS\dashrightarrow \cK(\cH)$ given by $\gamma_t(f)=f(t^{-1}B)$, where $B=C+D$ is an unbounded operator on $\cH$ with domain $\fS(\bR)$. Note that $\gamma$ is asymptotically equivariant since it is asymptotically equivalent to the asymptotically equivariant $\alpha\circ \beta_1$. Since each $\gamma_t$ is a $*$-homomorphism, the asymptotic morphism $\gamma$ is homotopy equivalent to the $*$-homomorphism $\gamma_1$. We will show that the class of $\gamma_1$ in equivariant $E$-theory is the classs of the identity.

By \cite{HK1}*{Corollary 15}, there is an orthonormal eigenbasis of $B$ consisting of Schwartz-class functions. Moreover, $\ker(B)$ has dimension $1$ and the nonzero eigenvalues of $B$ are $\pm\sqrt{2n}$ for $n\geq 1$. Let $p\in\cB(\cH)$ be the projection onto the kernel of $B$ and consider the following homotopy $H:\cS\to C([0,1],\cK(\cH))$:
\[H(f, s)=\begin{cases}
    f(s^{-1}B) & \text{for $s>0$,} \\
    f(0)p & \text{for $s=0$.}
\end{cases}\]
     To prove that $H$ is continuous at $s=0$, use that $f$ vanishes at $\infty$ and that $|\lambda|\geq \sqrt{2}$ for every nonzero eigenvalue $\lambda$ of $B$. This $H$ is an homotopy between $\gamma_1$ and a projection onto a $1$-dimensional subspace, that represents the identity in $E$-theory.
\end{proof}

As a corollary we obtain:
\begin{theorem}
    The Baum-Connes conjecture with coefficients holds for the infinite dihedral group. That is, the assembly map
    $\mu_r:K^{top}(D_\infty,A)\to K(C^*_r(D_\infty,A))$
    is an isomorphism for every separable $D_\infty$-$C^*$-algebra $A$.
\end{theorem}

By \cite{kranz}, the assembly map of the above theorem can be identified with the one defined by Davis and L\"uck in \cite{dl}. For $A=\C$, the left hand side of the Davis-L\"uck assembly was computed in \cite{SG}. We have $K_0(C_r^*(D_\infty))\cong \Z\oplus\Z\oplus \Z$ and $K_1(C_r^*(D_\infty))=0$.

\end{document}